\documentclass[12pt,amstex,amssymb]{amsart}
\usepackage{amsmath,amsthm,amsfonts,amssymb,amscd}
\usepackage[latin1]{inputenc}
\usepackage[all]{xy}
\usepackage[dvips]{graphicx}
\usepackage{amsmath}
\usepackage{amsthm}
\usepackage{amsfonts}
\usepackage{amssymb}

\newtheorem{Theorem}{Theorem}[section]
\newtheorem{Lemma}[Theorem]{Lemma}

\newtheorem{Proposition}[Theorem]{Proposition}

\def\fa{{\mathcal{F}}}

\def\vr{{\varphi}}

\def\ov{\overline}

\def\bc{{\mathbb{C}}}
\def\bn{{\mathbb{N}}}

\def\GL{\operatorname{{GL}}}

\def\dim{\operatorname{{dim}}}

\def\Diff{\operatorname{{Diff}}}

\def\Hol{\operatorname{{Hol}}}

\def\Id{\operatorname{{Id}}}

\def\Diff{\operatorname{{Diff}}}

\def\SO(3){\operatorname{{SO(3)}}}

\def\bc{\operatorname{{\mathbb C}}}
\def\C{\operatorname{{\mathbb C}}}
\def\R{\operatorname{{\mathbb R}}}
\def\Hol{\operatorname{{Hol}}}
\def\fa{\operatorname{{\mathcal F}}}
\def\Diff{\operatorname{{Diff}}}

\def\ov\bc{\operatorname{\overline{\mathbb{C}}}}

\title[A measurable stability theorem]{A measurable stability theorem for  holomorphic foliations transverse to fibrations}

\begin{document}

\author{Bruno Scardua}
\address{B. Scardua. Instituto de Matematica - Universidade Federal do Rio de Janeiro,
CP. 68530-Rio de Janeiro-RJ, 21945-970 - Brazil}
\email{scardua@im.ufrj.br}

\maketitle

\begin{abstract}
We prove that a transversely holomorphic foliation which is
transverse to the fibers of a  fibration, is  a Seifert fibration if
 the set of compact leaves is
not of zero measure. Similarly, we prove that a finitely generated
 subgroup of holomorphic diffeomorphisms of a connected complex
 manifold, is finite provided that the set of periodic orbits is
not of zero measure.
 \end{abstract}

\section{Introduction}
Foliations transverse to fibrations are among the very first and
simplest constructible examples of foliations, accompanied by a
well-known transverse structure\footnote{2000 Mathematics Subject
Classification: Primary 32S65, 57R30, 58E05; Secondary 32M05.}. These foliations are suspensions of
groups of diffeomorphisms and their behavior is closely related to
the action of the group in the fiber\footnote{Key words and phrases:  Holomorphic foliation, global holonomy, stable leaf.}. For these reasons, many
results holding for foliations in a more general context, are first
established for suspensions, i.e., foliations transverse to a
fibration. In this paper we pursue this idea, but  not restricted to
it. We investigate versions of the classical Stability theorems of Reeb (\cite{camacho, Godbillon}), regarding the behavior of the foliation in a neighborhood of a compact leaf,
replacing the finiteness of the holonomy group of the leaf by the
existence of a sufficient number of compact leaves. This is done for transversely holomorphic
(or transversely analytic) foliations.

  Let $\eta=(E, \pi, B,
F)$ be a (locally trivial) fibration with total space $E$, fiber $F$, base $B$
and projection $\pi\colon E \to B$. A foliation $\fa$ on $E$ is
{\it transverse to $\eta$} if: (1) for each $p \in E,$ the leaf
$L_p$ of $\fa$ with $p\in L_p$ is transverse to the fiber
$\pi^{-1}(q),$ $q=\pi(p);$ (2) $\dim(\fa) + \dim(F) = \dim(E);$
and (3) for each leaf $L$ of $\fa,$ the restriction $\pi|_L: L \to
B$ is a covering map. A theorem of Ehresmann (\cite{camacho} Ch. V)
\cite{Godbillon}) assures that if the fiber $F$ is compact, then conditions
(1) and (2) together already imply (3). Such foliations are conjugate to
suspensions and are characterized by their  {\it global holonomy}
(\cite{camacho}, Theorem 3, p. 103 and \cite{Godbillon}, Theorem 6.1, page 59).

The codimension one case is studied in  \cite{scardua1}.
In  \cite{scardua2} we study the case where the
ambient manifold is a hyperbolic complex manifold. In \cite{Fabio-Scardua} the authors prove
a natural version of the stability theorem of Reeb for (transversely holomorphic) foliations transverse to fibrations.  A foliation $\fa$ on $M$ is called a
{\it Seifert fibration} if all leaves are compact with finite
holonomy groups. \vglue.1in

The following stability theorem is proved in \cite{Fabio-Scardua}:

\begin{Theorem}
\label{Theorem:stabilitygeneralI} Let $\fa$ be a holomorphic
foliation transverse to a fibration $\pi\colon E
\overset{F}{\longrightarrow}~B$ with  fiber $F$.  If $\fa$ has a
compact leaf with finite holonomy group then $\fa$ is a Seifert
fibration.
\end{Theorem}

It is also observed in \cite{Fabio-Scardua} that the existence of a
trivial holonomy compact leaf is assured if $\fa$ is of codimension
$k$, has a compact leaf and the base $B$ satisfies $H^1(B,\R)=0,\
H^1(B, \GL(k,\C))=0.$

 Since a foliation transverse to a fibration  is conjugate
to a suspension of a group of diffeomorphisms of the fiber, we can
rely on the global holonomy of the foliation. As a general fact that
holds also for smooth foliations, if the global holonomy group is
finite then the foliation is a Seifert fibration.
The proof of Theorem~\ref{Theorem:stabilitygeneralI} relies on the
Local stability theorem of Reeb (\cite{camacho, Godbillon}) and the
following remark derived from classical theorems of Burnside and
Schur on finite exponent groups and periodic linear groups
(\cite{Fabio-Scardua}): {\it Let $G$ be a finitely generated
subgroup of holomorphic diffeomorphisms of a connected complex
manifold $F$. If each element of $G$ has finite order, then the
subgroups with a common fixed point are finite.}

\vglue.1in

In this paper we improve the above by proving  the following
theorems:

\begin{Theorem}
\label{Theorem:measuremainfoliations} Let $\fa$ be a transversely
holomorphic foliation transverse to a fibration $\pi\colon E
\overset{F}{\longrightarrow}~B$ with  fiber $F$ a connected complex
manifold. Denote by $\Omega(\fa)\subset E$ the subset of compact
leaves of $\fa$. Suppose that for some  regular volume measure $\tilde\mu$
on $E$ we have $\tilde\mu(\Omega(\fa))> 0$. Then $\fa$ is a Seifert
fibration with finite global holonomy.
\end{Theorem}

Parallel to this result we have the following version for groups:

\begin{Theorem}
\label{Theorem:maingroups}
 Let $G$ be a finitely generated subgroup
of holomorphic diffeomorphisms of a complex connected manifold $F$.
Denote by $\Omega(G)$ the subset of points $x \in F$ such that the
$G$-orbit of $x$ is periodic. Assume that for some regular volume measure
$\mu$ on $F$ we have $\mu(\Omega(G))>0$. Then $G$ is a finite group.
\end{Theorem}

As an immediate corollary of the above result we get that, for a
finitely generated subgroup $G\subset \Diff(F)$ of a complex
connected manifold $F$, if the volume of the orbits gives an
integrable function for some regular volume measure on $F$
 then all orbits
are periodic and the group is finite. This is related to results in \cite{scarduatohoku}.

\section{Holonomy and global holonomy}

Let $\fa$ be a codimension $k$ transversely holomorphic foliation
transverse to a fibration $\pi\colon E
\overset{F}{\longrightarrow}~B$ with fiber $F$, base $B$ and total
space $E$. We always assume that $B, F$ and $E$ are connected
manifolds. The  manifold $F$ is a complex manifold. Given a point
$p\in E$, put $b=\pi(p)\in B$ and denote by $F_b$ the fiber
$\pi^{-1}(b)\subset E$, which is biholomorphic to $F$. Given a point
$p\in E$ we denote by $\Hol(L_p)$ the holonomy group of the leaf
$L_p$ through $p$. This is a conjugacy class of equivalence, we
shall denote by $\Hol(L_p,F_b,p)$ its representative given by the
local representation of this holonomy calculated with respect to the
local transverse section induced by $F_b$ at the point $p\in F_b$.
The group  $\Hol(L_p, F_b,p)$ is therefore a subgroup of the group
of germs $\Diff(F_b,p)$ which is identified with the group
$\Diff(\bc^k,0)$ of germs at the origin $0\in \mathbb C^k$ of
complex diffeomorphisms.

Let  $\varphi\colon\pi_1(B,b) \to \Diff(F)$ be the global holonomy
representation of the  fundamental group of $B$ in the group of
holomorphic diffeomorphisms of the manifold $F$, obtained by lifting
closed paths in $B$ to the leaves of $\fa$ via the covering maps
$\pi\big|_L\colon L \to B$, where $L$ is a leaf of $\fa$. The image
of this representation is the {\it global holonomy} $\Hol(\fa)$ of
$\fa$ and its construction shows that $\fa$ is conjugated to the
suspension of its global holonomy  (\cite{camacho}, Theorem 3, p.
103). Given a base point $b\in B$ we shall denote by  $\Hol(\fa,
F_b) \subset \Diff(F_b)$ the representation of the global holonomy
of $\fa$ based at $b$. \vglue.1in

 From the classical theory (\cite{camacho}, chapter V)  and \cite{Fabio-Scardua} we have:

\begin{Proposition}
\label{Proposition:globalholonomyisotropy} Let $\fa$ be a foliation
on $E$ transverse to the
fibration $\pi\colon E \to B$ with fiber $F$. Fix a point $p\in E$, $b=\pi(p)$
 and denote by  $L$ the leaf that contains $p$.
\begin{enumerate}
\item The  holonomy group $\Hol(L, F_b, p)$ of $L$
 is the subgroup of the global holonomy $\Hol(\fa,
F_b) \subset \Diff(F_b)$ of those  elements that have $p$ as a fixed
point.
\item Given another intersection point $q\in L\cap F_b$ there is a
global holonomy map $h\in \Hol(\fa, F_b)$
such that $h(p)=q$.

\item Suppose that  the global holonomy $\Hol(\fa)$ is finite. If $\fa$
has a compact leaf then it is a Seifert fibration, i.e., all leaves
are compact with finite holonomy group.

\item  If  $\fa$ has a compact  leaf $L_0\in\fa$ then
each point $p \in F_b\cap L_0$   has periodic orbit in the global
holonomy $\Hol(\fa)$. In particular there are $\ell\in\mathbb N$ and
$p\in F$ such that  $h^\ell(p)=p$ for every $h\in \Hol(\fa)$.

\end{enumerate}

\end{Proposition}

\section{Periodic groups and groups of finite exponent}
\label{section:generalities}

First we recall some facts from the theory of Linear groups. Let $G$
be a group with identity $e_G\in G$. The group is {\it periodic} if
 each element of $G$ has finite order.
 A periodic group $G$ is {\it periodic of bounded exponent} if
there is an uniform upper bound for the orders of its elements. This
is equivalent to the existence of $m \in \mathbb N$ with $g^m = 1$
for all $g \in G$ (cf. \cite{Fabio-Scardua}).  Because of this, a
group which is periodic of bounded exponent is also called a group
of {\it finite exponent}.
 Given  $R$  a ring with identity, we say that a group $G$ is $R$-{\it
linear\/} if it is isomorphic to a subgroup of the matrix group
$\GL(n,R)$ (of $n \times n$ invertible matrices with coefficients
belonging to $R$) for some $n \in \bn$.  We will consider  complex
linear groups. The following classical results are  due to Burnside and Schur.

\begin{Theorem}
\label{Theorem:Burnside}  Regarding periodic groups and groups of
finite exponent we have:
\begin{enumerate}

\item {\rm(Burnside, 1905 \cite{Burnside})} A {\rm(}not
necessarily finitely generated{\rm)}
 complex linear group $G \subset \GL(k,\bc)$  of finite exponent
$\ell$ has finite order; actually we have $|G| \le \ell^{k^2}$.

\item {\rm(Schur, 1911 \cite{Schur})} Every finitely generated periodic subgroup of
$\GL(n,\mathbb C)$ is finite.

\end{enumerate}
\end{Theorem}
\noindent Using these results we  obtain:

\begin{Lemma}[\cite{Fabio-Scardua} Lemmas~2.3, 3.2 and 3.3]
\label{Lemma:finiteexponentgerms} About periodic groups of germs of
complex diffeomorphisms we have:
\begin{enumerate}

\item A finitely
generated periodic subgroup  $G \subset \Diff(\bc^k,0)$ is
necessarily finite.

\item  A {\rm(}not necessarily finitely generated{\rm)} subgroup
 $G \subset \Diff(\bc^k,0)$ of finite exponent is necessarily finite.

\item  Let $G\subset \Diff(\mathbb C^k,0)$ be a finitely
generated subgroup.  Assume that there is an invariant connected
neighborhood $W$ of the origin in $\mathbb C^k$ such that each point
$x$ is periodic for each element $g \in G$. Then $G$ is a finite
group.

\item Let $G \subset \Diff(\bc^k,0)$ be a {\rm(}not necessarily finitely generated{\rm)} subgroup
such that for each point $x$ close enough to the origin,  the
pseudo-orbit of $x$ is finite of {\rm(}uniformly bounded{\rm)} order $\le \ell$ for some
$\ell \in \bn$, then $G$ is finite.

\end{enumerate}

\end{Lemma}

Given a subgroup $G\subset \Diff(F)$ and a point $p\in F$ the {\it
stabilizer} of $p$ in $G$  is the subgroup $G(p)\subset G$ of the
elements $f \in G$ such that $f(p)=p$. From the above we have:

\begin{Proposition}
\label{Proposition:groupdiffeofiberfinite}

Let $G\subset \Diff(F)$ be a {\rm(}not necessarily finitely
generated{\rm)} subgroup of holomorphic diffeomorphisms of a connected complex
manifold $F$.
\begin{enumerate}

\item If $G$ is periodic and finitely generated or $G$ is
 periodic of finite exponent, then each stabilizer subgroup of $G$ is
finite.

\item Assume that there is a point $p\in F$ which is fixed by $G$ and
a fundamental system of neighborhoods $\{U_\nu\}_\nu$ of $p$ in $F$
such that each $U_\nu$ is invariant by $G$,  the orbits of $G$ in $U
_\nu$ are periodic {\rm(}not necessarily with uniformly bounded
orders{\rm)}. Then $G$ is a finite group.

\item Assume that $G$ has a periodic orbit $\{x_1,...,x_r\}\subset F$
such that for each $j\in \{1,...r\}$ there is a fundamental system
of neighborhoods $U_\nu ^j$ of $x_j$ with the property that $U_\nu =
\bigcup \limits_{j=1} ^r U_\nu ^j $ is invariant under the action of
$G$, $U_\nu ^j \cap U_\nu ^\ell = \emptyset$ if $j \ne \ell$ and
each orbit in $U_\nu$ is periodic. Then $G$ is periodic.

\end{enumerate}
\end{Proposition}

\begin{proof}
 In order to prove (1) we
consider the case where $G$ has a fixed point $p\in F$. We identify
the group $\mathcal G_p$ of germs at $p$ of maps in $G$ with a
subgroup of $\Diff(\mathbb C^n,0)$ where $n=\dim F$. If $G$ is
finitely generated and periodic, the group $\mathcal G_p$ is
finitely generated and periodic. By
Lemma~\ref{Lemma:finiteexponentgerms} (1) the group $\mathcal G_p$
is finite and by the Identity principle the group $G$ is also finite
of same order. If $G$ is periodic of finite exponent then the group
$\mathcal G_p$ is periodic of finite exponent. By
Lemma~\ref{Lemma:finiteexponentgerms} (2) the group $\mathcal G_p$
is finite and by the Identity principle the group $G$ is also finite
of same order. This proves (1).

As for (2), since $U_\nu$ is $G$-invariant each element $g\in G$
induces by restriction to $U_\nu$ an element of a group
$G_{\nu}\subset \Diff(U_\nu)$. It is observed in
\cite{Fabio-Scardua} (proof of Lemma~3.5) that the finiteness of the orbits in $U_\nu$
implies that $G_{\nu}$ is periodic. By the Identity principle, the
group $G$ is also periodic of the same order. Since $G=G(p)$,  (2)
follows from (1). (3) is proved like the first part of (2).
\end{proof}

The following simple remark gives the finiteness of finite exponent
 groups of holomorphic diffeomorphisms having a periodic
orbit.

\begin{Proposition}[Finiteness lemma]
\label{Proposition:finitenesslemma}
Let $G$ be a  subgroup of holomorphic
diffeomorphisms of a connected complex manifold $F$. Assume that:

\begin{enumerate}

\item $G$ is periodic of finite exponent or $G$ is finitely generated and periodic.

\item $G$ has a finite orbit in $F$.

\end{enumerate}
Then $G$ is finite.
\end{Proposition}

\begin{proof}
Fix a point $x \in F$ with finite orbit  we can  write $\mathcal
O_G(x)=\{x_1,...,x_k\}$ with $x_i\ne x_j$ if $i \ne j$. Given any
 diffeomorphism $f \in G$ we have
$\mathcal O_G(f(x))=\mathcal O_G(x)$ so that there exists an unique
element $\sigma \in S_k$ of the symmetric group such that
$f(x_j)=x_{\sigma_f(j)}, \, \forall j =1,...,k$\,. We can therefore
define a map
\[
\eta \colon G \to S_k, \, \eta(f)=\sigma_f\,.
\]
Now, if $f,g \in G$ are such that $\eta(f)=\eta(g)$, then
$f(x_j)=g(x_j), \, \forall j$ and therefore $h=f \, g^{-1}\in G$
fixes the points $x_1,...,x_k$. In particular $h$ belongs to the
stabilizer $G_x$. By
Proposition~\ref{Proposition:groupdiffeofiberfinite}   (1) and (2)
(according to $G$ is finitely generated or not) the group  $G_x$ is
finite. Thus, the  map $\eta \colon G \to S_k$ is a finite map.
Since $S_k$ is a finite group this implies that $G$ is finite as
well.

\end{proof}

\section{Measure and finiteness}

The following lemma paves the way to
Theorems~\ref{Theorem:measuremainfoliations} and
\ref{Theorem:maingroups}.

\begin{Lemma}
\label{Lemma:measuregroup} Let $G$ be a subgroup of complex
diffeomorphisms  of a connected complex manifold $F$.
 Denote by $\Omega(G)$ the set of points $x \in F$ such that
the orbit $\mathcal O_G(x)$ is periodic.  Let $\mu$ be a regular volume measure on $F$. If
$\mu(\Omega(G))>0$  then $G$ is a
periodic group of finite exponent.
\end{Lemma}
\begin{proof}

We have $\Omega(G)=\{x \in F: \# \mathcal O_G(x)< \infty\}=
\bigcup\limits_{k=1} ^\infty \{x \in F: \# \mathcal O_G(x)\leq
k\}$, therefore there is some $k\in \mathbb N$ such that
\[
\mu(\{x \in F: \# \mathcal O_G(x) \leq k\})>0.
\]
In particular, given any diffeomorphism $f \in G$ we have
\[
\mu(\{x \in F: \# \mathcal O_f(x) \leq k\})>0.
\]
In particular, there is $k_f \leq k$ such that the set $X=\{x \in F:
f^{k_f}(x)=x\}$ has positive measure. Since $X\subset F$ is an
analytic subset, this implies that $X=F$ (a proper analytic subset
of a connected complex  manifold has zero measure). Therefore, we
have $f^{k_f}=\Id$ in $F$. This shows that $G$ is periodic of
finite exponent.
\end{proof}

\begin{proof}[Proof of Theorem~\ref{Theorem:measuremainfoliations}]
Fix a base point $b \in B$. Denote by $\mu$ the regular volume measure on
the  fiber $F_b$ induced by the measure
 $\tilde \mu$. By Proposition~\ref{Proposition:globalholonomyisotropy}
 the compact leaves correspond
  to periodic orbits of the global holonomy $\Hol(\fa, F_b)$. Therefore,
  by the hypothesis the global holonomy $G=\Hol(\fa, F_b)$ and the measure
  $\mu$ satisfy the hypothesis of
Lemma~\ref{Lemma:measuregroup}.
  By this lemma the global holonomy is periodic of finite exponent. Since
  this group has some periodic orbit, by the Finiteness lemma
  (Proposition~\ref{Proposition:finitenesslemma}) the global holonomy group is finite.
  By Proposition~\ref{Proposition:globalholonomyisotropy} (3)
  the foliation is a Seifert fibration.
\end{proof}

The construction of the suspension of a group action gives
Theorem~\ref{Theorem:maingroups} from
Theorem~\ref{Theorem:measuremainfoliations}.

\begin{proof}[Proof of Theorem~\ref{Theorem:maingroups}]
Since $G$ is finitely generated, there are a compact connected
manifold $B$ and a representation $\vr \colon \pi_1(B) \to \Diff(F)$
such that the image $\vr(\pi_1(B))=G$. The manifold $B$ is not
necessarily a complex manifold, but this makes no difference in our
argumentation based only on the fact that the foliation is
transversely holomorphic. Denote by $\fa$ the suspension foliation
of the fibre bundle $\pi\colon E \to B$ with fiber $F$ which has
global holonomy conjugate to $G$. The periodic orbits of $G$ in $F$
correspond in a natural way to the
 leaves of $\fa$ which have finite order with respect to the
fibration $\pi \colon E \to B$, i.e., the leaves which  intersect the
fibers of $\pi \colon E \to B$ only at a finite number of points. Thus,
because the basis is compact,
each such leaf (corresponding to a finite orbit of $G$) is compact.
The measure $\mu$ induces a natural regular volume  measure $\tilde\mu$ on $E$.
Therefore by the hypothesis, we have
$\tilde\mu(\Omega(\fa))>0$. By Theorem~\ref{Theorem:measuremainfoliations}
the global holonomy $\Hol(\fa)$ is
finite. Thus the group $G$ is finite.
\end{proof}

\end{document}